\documentclass[10pt]{amsart}

\usepackage{amsmath, amsfonts, amsthm, amssymb, graphicx, fullpage, enumerate}

\newtheorem{theorem}{Theorem}      
\newtheorem{lemma}{Lemma}
\newtheorem{corollary}[lemma]{Corollary}

\newtheorem{proposition}[lemma]{Proposition}     

\newtheorem*{main theorem}{Main Theorem}  
\newtheorem*{thmA}{Theorem A}  
\newtheorem*{thmB}{Theorem B}    
\newtheorem*{thmC}{Theorem C}   
\newtheorem*{thmD}{Theorem D}  
\newtheorem*{thmE}{Theorem E}    
\newtheorem*{thmF}{Theorem F}    
\newtheorem*{thmG}{Theorem G}     
\theoremstyle{remark} 
\newtheorem*{rem}{Remark} 
\theoremstyle{definition}  
\newtheorem{definition}{Definition}  
     
\def\T{\mathbb{T}}         
\def\N{\mathbb{N}}     
          
\def\R{\mathcal{R}}     
\def\Z{\mathbb{Z}}     
\def\C{\mathbb{C}}  
\def\P{\mathcal{P}}

\begin{document}
\title{S\'ark\"ozy's Theorem for $\P$-intersective Polynomials}
\author{Alex Rice }
\address{Department of Mathematics, The University of Georgia, Athens, GA 30602, USA}
\email{arice@math.uga.edu} 
\subjclass[2000]{11B30}
\begin{abstract} We define a necessary and sufficient condition on a polynomial $h\in \Z[x]$ to guarantee that every set of natural numbers of positive upper density contains a nonzero difference of the form $h(p)$ for some prime $p$. Moreover, we establish a quantitative estimate on the size of the largest subset of $\{1,2,\dots,N\}$ which lacks the desired arithmetic structure, showing that if $\deg(h)=k$, then the density of such a set is at most a constant times $(\log N)^{-c}$ for any $c<1/(2k-2)$.  We also discuss how an improved version of this result for $k=2$ and a relative version in the primes can be obtained with some additional known methods.
\end{abstract}
\maketitle

\setlength{\parskip}{5pt} 
\section{Introduction}
\subsection{Background} A set $A\subseteq \N$ is said to have \textit{positive upper density} if
\begin{equation*} \limsup_{N \to \infty} \frac{|A\cap[1,N]|}{N}>0, \end{equation*} where $[1,N]$ denotes $\{1,2,\dots,N\}$. In the late 1970s, S\'ark\"ozy and Furstenberg independently confirmed a conjecture of Lov\'asz that any set of natural numbers of positive upper density necessarily contains two elements which differ by a perfect square. Furstenberg \cite{Furst} used ergodic theory and obtained a purely qualitative result, proving the conjecture exactly as stated above. S\'ark\"ozy, however, employed a Fourier analytic density increment strategy, inspired by Roth's proof of the analogous conjecture for three-term arithmetic progressions \cite{Roth}, to prove the following quantitative strengthening. 

\begin{thmA}[S\'ark\"ozy, \cite{Sark1}]  If $A\subseteq [1,N]$ and $n^2 \notin A-A$ for all $n\in \N$, then\begin{equation*} \label{sarkb} \frac{|A|}{N} \ll \Big(\frac{(\log \log N)^2}{\log N}\Big)^{1/3}. \end{equation*}
\end{thmA}  
  
In this and the following theorems,  $A-A$ denotes the difference set $\{a-a':a,a' \in A\}$, the symbol $\ll$ denotes ``less than a constant times'', and we implicitly assume that $N$ is large enough to make the right hand side of the inequalities defined and positive.  An extensive literature has been developed on improvements and extensions of Theorem A, for which the reader may refer to \cite{PSS}, \cite{BPPS}, \cite{Slip}, \cite{Lucier}, \cite{LM}, \cite{Le}, and \cite{HLR}. In the same series of papers, S\'ark\"ozy answered a similar question of Erd\H{o}s concerning shifted primes.
\begin{thmB}[S\'ark\"ozy, \cite{Sark3}] \label{sarkP}  If $A\subseteq [1,N]$ and $p-1 \notin A-A$ for all primes $p$, then\begin{equation}\label{sarpb} \frac{|A|}{N} \ll \frac{(\log\log\log N)^3\log\log\log\log N}{(\log \log N)^2}. \end{equation}
\end{thmB}

\noindent The bounds in Theorem B have been improved, first by Lucier \cite{Lucier2} and later by Ruzsa and Sanders \cite{Ruz}, who replaced (\ref{sarpb}) with $|A|/N \ll e^{-c(\log N)^{1/4}}$.  

A natural generalization of Theorem A is the replacement of the squares with the image of a more general polynomial. However, to hope for such a result for a given  polynomial $h \in \Z[x]$, it is clearly necessary that $h$ has a root modulo $q$ for every $q \in \N$, as otherwise there would be a set $q\N$ of positive density with no differences in the image of $h$. It follows from a theorem of Kamae and Mend\`es France \cite{KMF} that this condition is also sufficient, in a qualitative sense, and in this case we say that $h$ is an \textit{intersective polynomial}. \\

Examples of intersective polynomials include any polynomial with an integer root and any polynomial with two rational roots with coprime denominators. However, there are also intersective polynomials with no rational roots, for example $(x^3-19)(x^2+x+1)$. The best current bounds for this most general setting are essentially due to Lucier, who successfully adapted the density increment procedure by utilizing $p$-adic roots and allowing the polynomial to change at each step of the iteration.

\begin{thmC}[Lucier, \cite{Lucier}] Suppose $h\in \Z[x]$ is an intersective polynomial of degree $k\geq 2$ with positive leading term. If $A \subseteq [1,N]$ and $h(n) \notin A-A$ for all $n \in \N$ with $h(n)> 0$, then \begin{equation*} \frac{|A|}{N} \ll \Big(\frac{(\log\log N)^\mu}{\log N}\Big)^{1/(k-1)}, \quad \mu=\begin{cases}3 &\text{if }k=2 \\ 2 &\text{if }k>2 \end{cases},\end{equation*} where the implied constant depends only on $h$.
\end{thmC}
 
In \cite{thesis}, the author made an extremely mild improvement to Theorem C, showing that one can in fact take $\mu=1$. By the symmetry of difference sets, Theorem C and all the following theorems clearly imply the analogous results for the negative values of a polynomial with negative leading term. Recently, Hamel, Lyall, and the author utilized Lucier's techniques in extending the best known bound on the size of a set with no square differences, due originally to Pintz, Steiger, and Szemer\'edi \cite{PSS} and extended to $k^{\text{th}}$ powers by Balog, Pelikan, Pintz, and Szemer\'edi \cite{BPPS}, to all intersective polynomials of degree $2$, which is to say quadratic polynomials which have two rational roots with coprime denominators.

\begin{thmD}[Hamel, Lyall, Rice, \cite{HLR}] Suppose $h\in \Z[x]$ is an  intersective quadratic polynomial with positive leading term. If $A\subseteq [1,N]$ and $h(n) \notin A-A$ for all $n\in \N$ with $h(n)>0$, then 
\begin{equation*}\frac{|A|}{N} \ll (\log N)^{-c \log\log\log\log N}
\end{equation*}for any $c<1/\log 3$, where the implied constant depends only on $h$ and $c$. 
\end{thmD}

Some work has also been done to combine extensions of Theorem A with Theorem B. Li and Pan \cite{lipan} established the following quantitative result.

\begin{thmE}[Li, Pan, \cite{lipan}]Suppose $h \in \Z[x]$ has positive leading term and $h(1)=0$. If $A \subseteq [1,N]$ and $h(p) \notin A-A$ for all primes $p$ with $h(p)> 0$, then \begin{equation*} \frac{|A|}{N} \ll 1/\log\log\log N. \end{equation*}
\end{thmE}

Additionally, L\^e  and Li-Pan  applied transference principles inspired by work of Green and Tao (\cite{Green}, \cite{GrTao}, \cite{GrTao2}) to prove analogs of Theorems C and E, respectively, for dense subsets of the primes, which we denote by $\P$. We state the qualitative results below.

\begin{thmF}[L\^e, \cite{Le}] If $h\in \Z[x]$ is an intersective polynomial with positive leading term, $A\subseteq \P$, and 
\begin{equation} \label{relden} \limsup_{N\to \infty} \frac{|A\cap[1,N]|}{|\P\cap[1,N]|} >0 ,
\end{equation} then there exist $a,a'\in A$ and $n\in \N$ with $a-a'=h(n)>0$.

\end{thmF}

\noindent If $A\subseteq \P$ meets condition (\ref{relden}), we say that it has \textit{positive relative upper density} in the primes.

\begin{thmG}[Li, Pan, \cite{lipan}] If $h \in \Z[x]$ has positive leading term, $h(1)=0$, and $A\subseteq \P$ has positive relative upper density, then there exist $a,a' \in A$ and $p\in \P$ with $a-a'=h(p)>0$.
\end{thmG}

\subsection{Main result of this paper} Just as there are intersective polynomials without integer roots, it seems natural to think that a result like Theorem E should hold for a larger class of polynomials. A moments consideration indicates that the correct analog to the intersective condition on a polynomial $h$ when looking for differences of the form $h(p)$ is to insist that $h$ not only has a root modulo $q$ for every $q \in \N$, but has a root at a congruence class that admits infinitely many primes, leading to the following definition.

\begin{definition}  A polynomial $h \in \Z[x]$ is called {\it $\P$-intersective} if for every $q \in \N$, there exists $r \in \Z$ such that $(r,q)=1$ and $q \mid h(r)$. Equivalently, for every $p\in \P$, there exists $z_p \in \Z_p$, where $\Z_p$ denotes the $p$-adic integers, such that $h(z_p)=0$ and $z_p \not\equiv 0 \mod p$.
\end{definition}
\begin{rem}After the initial version of this paper was posted on arXiv server, the author learned that in \cite{Le2}, a survery on problems and results on intersective sets, Th\'ai Ho\`ang L\^e independently posed the same question and arrived at the same definition, which he coined \textit{intersective polynomials of the second kind}. Even later, the author learned that these polynomials were considered by Wierdl \cite{wierdl} in his thesis, where he called them \textit{intersective along the primes}.  
\end{rem}
Examples of $\P$-intersective polynomials include any polynomial with a root at $1$ or $-1$, any polynomial with two rational roots $a/b$ and $c/d$ such that $(ab,cd)=1$, and presumably lots more. The necessity of this condition is almost as clear as that of the original intersective condition. To exhibit this, suppose we have $h\in \Z[x]$ and $q\in \N$ such that the only roots of $h$ modulo $q$ share common factors with $q$. In particular, there are finitely many primes $p$ such that $q\mid h(p)$. Letting $m=\max\{h(p)/q: p\in \P, \ q\mid h(p)\}$ if such primes exist and $m=0$ otherwise, we see that $q(m+1)\N$ is a set of positive upper density which contains no differences of the form $h(p)$. Wierdl \cite{wierdl} observed in his thesis that one can again deduce the sufficiency of this condition, in a qualitative sense, from the aforementioned theorem of Kamae and Mend\'es France \cite{KMF}, and here we borrow heavily from \cite{Lucier}, \cite{Ruz}, \cite{LM}, and \cite{lipan} to establish the following quantitative result.

\begin{theorem} \label{pint} Suppose $h\in \Z[x]$ is a $\P$-intersective polynomial of degree $k\geq 2$ with positive leading term. If $A\subseteq [1,N]$ and $h(p) \notin A-A$ for all $p\in \P$ with $h(p) > 0$, then 
\begin{equation}\label{bound}\frac{|A|}{N} \ll (\log N)^{-c}\end{equation} 
for any $c<1/(2k-2)$, where the implied constant depends only on $h$ and $c$. 
\end{theorem}

\noindent In fact, with a few careful modifications one can make this arbitrarily small power of $\log N$ more explicit, but here we stick to the slightly less precise version for a more pleasing exposition. 

\subsection{Additional Results}
In addition to Theorem \ref{pint}, one can conclude the following analogs of previous results from the estimates we obtain along the way.

\begin{theorem} \label{pintQ} Suppose $h\in \Z[x]$ is a  $\P$-intersective quadratic polynomial with positive leading term. If $A \subseteq[1,N]$ and $h(p) \notin A-A$ for all $p \in \P$ with $h(p)>0$, then  
\begin{equation*} \frac{|A|}{N} \ll (\log N)^{-c \log\log\log\log N}
\end{equation*}for any $c<1/2\log 3$, where the implied constant depends only on $h$ and $c$. 
\end{theorem}
\noindent Just as in the traditional setting, the $\P$-intersective condition is greatly simplified when restricted to degree $2$, as a quadratic polynomial is $\P$-intersective if and only if it has rational roots $a/b$ and $c/d$ with $(ab,cd)=1$. 
\begin{theorem}\label{rel} If $h\in \Z[x]$ is a $\P$-intersective polynomial with positive leading term and $A\subseteq \P$ has positive relative upper density, then there exist $a,a'\in A$ and $p\in \P$ with $a-a'=h(p)>0$.
\end{theorem} Equipped with the techniques and results of this paper, the modifications of the arguments in \cite{HLR} and \cite{Le} required to obtain Theorems \ref{pintQ} and \ref{rel}, respectively, are so minor that we do not provide the full details here. Alternatively, we discuss the required adaptations informally in Appendix \ref{C}.

\begin{center} \textbf{Acknowledgement} \end{center}
The author would like to thank his thesis advisor, Neil Lyall.
\section{Preliminaries} \label{prelim} 

To begin our effort to prove Theorem \ref{pint}, we fix a $\P$-intersective polynomial $h$ of degree $k\geq 2$ with positive leading term and an arbitrary $\epsilon>0$, and we set  $s=2^k+6$ and $K=2^{10k}$. For our current purposes, we only use that $s >9$, but the choice also plays a role in our discussion of Theorems \ref{pintQ} and \ref{rel}.  We also fix a natural number $N$ which, at the expense of constants in Theorem \ref{pint}, we are always free to insist is sufficiently large with respect to $h$ and $\epsilon$. For convenience, we take this as a perpetual hypothesis and abstain from including it further. We use the letters $C$ and $c$ to denote appropriately large or small positive constants, which will change from line to line and we allow, along with any implied constants, to depend on $h$ and $\epsilon$. 

\noindent Further, we fix a set $A \subseteq [1,N]$ with $|A|/N=\delta >0$ and set $Q(\delta) = \exp\Big(C_0\delta^{-(k+\epsilon-1)}\Big)$ for  a constant $C_0$. 

\subsection{Auxiliary Polynomials and Related Definitions} \label{defsec} We apply the modified density increment strategy described in \cite{Lucier}, which allows for the polynomial to change at each stage of the iteration. The following definitions describe all of the polynomials that we could potentially encounter, as well as several other objects that will appear in the argument.

\begin{definition}For each $p \in \P$, we fix $z_p \in \Z_p$ with $h(z_p)=0$ and $z_p \not\equiv 0 \mod p$. By reducing and applying the Chinese Remainder Theorem, the choices of $z_p$ determine, for each natural number $d$, a unique integer $r_d \in (-d,0]$, which consequently satisfies $d \mid h(r_d)$ and $(r_d,d)=1$. 

\noindent We define the function $\lambda$ on $\N$ by letting $\lambda(p)=p^m$, where $m$ is the multiplicity of $z_p$ as a root of $h$ in $\Z_p$, and then extending it to be completely multiplicative.

\noindent For each $d\in \N$, we define the \textit{auxiliary polynomial} $h_d$ by 
\begin{equation*} h_d(x)=h(r_d + dx)/\lambda(d).\end{equation*}

\noindent Lucier observed in Lemma 21 of \cite{Lucier} that each $h_d$ has integer coefficients, and it is important to note that the leading coefficients grow with $d$ at least as quickly, up to a constant, as the other coefficients.

\noindent Further, we let $\Lambda_{d}  =\{x \in \N: r_d+dx \in \P\}$, and for $L\in \N$ we define \begin{equation*}H_d=H_d(L)=\{ x\in \N : 0<h_d(x)<L/s\}
\end{equation*}
and $$M_d=M_d(L)= (L/sb_d)^{1/k},$$ where $b_d$ is the leading coefficient of $h_d$, noting that 
\begin{equation}\label{symdif} |[1,M_d] \triangle H_d | = O(1),
\end{equation} where $\triangle$ denotes the symmetric difference. We also define a function $\nu_d$ on $\Z$ by $$\nu_d(x)=\frac{\phi(d)}{d}\log(r_d+dx) 1_{\Lambda_d}(x),$$ where $\phi$ is the Euler totient function, and for a set $B\subseteq [1,L]$ we define $$\R_d(B) = \R_d(B,L)=\sum_{\substack{x\in \Z\\ y\in H_d}} 1_B(x)1_B(x+h_d(y))\nu_d(y).$$ In the definitions above, $L$ should always be replaced with the size of the appropriate ambient interval.
\end{definition} 
\subsection{Counting Primes in Arithmetic Progressions}
 
For $X,a,q\in \N$, we define 
\begin{equation*} \psi(X,a,q) = \sum_{\substack{p\in \P\cap[1,X] \\ p\equiv a \text{ mod } q}} \log p.
\end{equation*} 
The classical estimates on $\psi(X,a,q)$ come from the famous Siegel-Walfisz Theorem, which can be found for example in Corollary 11.19 of \cite{MV}.
\begin{lemma}[Siegel-Walfisz Theorem] \label{SW} If $q\leq (\log X)^D$ and $(a,q)=1$, then $$ \psi(X,a,q) = X/\phi(q) + O(Xe^{-c\sqrt{\log X}}) $$ for some constant $c=c(D)>0$.
\end{lemma}

Ruzsa and Sanders \cite{Ruz} established asymptotics for $\psi(X,a,q)$ for certain moduli $q$ beyond the  limitations of Lemma \ref{SW} by exploiting a dichotomy based on exceptional zeros, or lack thereof, of Dirichlet $L$-functions. In particular, the following result follows from their work.

\begin{lemma} \label{RS} If $ Q(\delta) \leq e^{c_1\sqrt{\log N}}$ for a sufficiently small constant $c_1=c_1(k)>0$, then there exist $q_0 \leq Q(\delta)^{3K}$ and $\rho \in [1/2,1)$ with $(1-\rho)^{-1} \ll q_0$ such that 
\begin{equation}\label{lb} \psi(X,a,q)= \frac{X}{\phi(q)}-\frac{\chi(a)X^{\rho}}{\phi(q)\rho}+O(Xe^{-30kK^2c_1\sqrt{\log X}}),
\end{equation}
where $\chi$ is a Dirichlet character modulo $q_0$, provided $X\geq N^{1/10k}$, $q_0 \mid q$, $(a,q)=1$,  and $q \leq (q_0Q(\delta))^{3K}$.
\end{lemma}  

\noindent Lemma \ref{RS} is a purpose-built special case of Proposition 4.7 of \cite{Ruz}, which in the language of that paper can be deduced by considering the pair $(Q(\delta)^{10K^2},Q(\delta)^{3K})$, where $q_0$ is the modulus of the exceptional Dirichlet character if the pair is exceptional and $q_0=1$ if the pair is unexceptional. 

\noindent  As remarked in the proof of Proposition 5.3 of \cite{Ruz}, the asymptotic in Lemma \ref{RS} implies that under the hypotheses we have 
\begin{equation}\label{psiB} \psi(X,a,q) \gg \frac{X}{\phi(q)}-\frac{X^{\rho}}{\phi(q)\rho} \geq (1-\rho)X/\phi(q) \gg \frac{X}{q_0\phi(q)}.
\end{equation}

\subsection{A Uniform Estimate on $\R_d$}We obtain Theorem \ref{pint} as a consequence of the following, stronger result, which says that the number of solutions to $a-a'=h(p)>0$ with $a,a' \in A$, $p\in \P$, has the correct order of magnitude. In addition, we obtain this estimate uniformly in $d$ for a range of auxiliary polynomials $h_d$, which serves as the primary input required to apply the techniques of \cite{Le} and conclude Theorem \ref{rel}.
\begin{theorem} \label{pint2} There exists a constant $C$ depending only on $h$, $\epsilon$, and $C_0$ such that 
\begin{equation*} \R_d(A) \geq \exp (-C\delta^{-(k+\epsilon-1)} ) NM_d
\end{equation*} for all $d \leq \max \{\log N, Q(\delta) \}$, provided $\delta \geq C(\log N)^{-1/2(k+\epsilon-1)}$.
\end{theorem}
 
\section{Main Iteration Lemma: Deducing Theorem \ref{pint2}}
We now make the assumption that 
\begin{equation} \label{assmp} Q(\delta) \leq e^{c_1\sqrt{\log N}} \end{equation} 
for a sufficiently small constant $c_1>0$, which is implied by the condition $\delta \geq C_0(\log N)^{-1/2(k+\epsilon-1)}/c_1$, and we fix $\rho$ and $q_0$ yielded by Lemma \ref{RS}. Also, we set $\gamma=k+\epsilon/2$ and for $d,L\in \N$ we define
\begin{equation*}\Psi_d=\Psi_d(L)=\phi(d)\psi(dM_d,r_d,d)/d,
\end{equation*}
noting by (\ref{psiB}) that for appropriate $d$ and $L$ we have
\begin{equation}\label{psiB2} \Psi_d \gg (1-\rho)M_d \gg M_d/q_0. 
\end{equation}
We deduce Theorem \ref{pint2}  from the following iteration lemma, which states that a set which is deficient in the desired arithmetic structure spawns a new, significantly denser subset of a slightly smaller interval with an inherited deficiency in the structure associated to an appropriate auxiliary polynomial.

\begin{lemma} \label{mainit} Suppose $B\subseteq [1,L]$, $|B|/L=\sigma\geq\delta$, and $L\geq \sqrt{N}$. If $q_0 \mid d$, $d/q_0\leq \max \{ \log N, Q(\delta)\}^2$, and
\begin{equation*} \R_d(B) \leq \sigma^2L\Psi_d/8, \end{equation*}
then there exists $q\ll \sigma^{-\gamma}$ and $B'\subseteq [1,L']$ with $L' \gg \sigma^{\gamma(k+1)}L,$ $\R_{qd}(B')\leq \R_d(B),$ and \begin{equation*} |B'|\geq (\sigma+c\sigma^{\gamma})L'. \end{equation*}
\end{lemma} 

The following proposition exhibits the aforementioned inheritance of deficiency in arithmetic structure, and is essential to the deduction of Theorem \ref{pint2} from Lemma \ref{mainit} as well as the proof of Lemma \ref{mainit} itself.

\begin{proposition} \label{scale} If $B \subseteq [1,L]$ and $B' \subseteq \{\ell \in [1,L'] : x+\ell \lambda(q)\in B \}$ for some $x\in\Z$, $q\in\N$, and $L'\leq L/\lambda(q)$, then for any $d\in\N$, $$\R_{qd}(B')\leq \R_d(B).$$ 
\end{proposition}

\begin{proof} Suppose $B\subseteq [1,L]$, $B' \subseteq \{\ell \in [1,L'] : x+\ell \lambda(q)\in B \},$ $L'\leq L/\lambda(q)$, and \begin{equation*} L'/s>\ell-\ell'=h_{qd}(n)=\frac{h(r_{qd}+qdn)}{\lambda(q)\lambda(d)}>0\end{equation*} for $\ell,\ell' \in B'$, $n \in \Lambda_{qd}$. Recalling that $r_{qd} \equiv r_d$ mod $d$, there is an integer $m$ such that $r_{qd}=r_d+md$,   so \begin{equation*}\ell-\ell' = \frac{h(r_d+d(m+qn))}{\lambda(q)\lambda(d)}=\frac{h_d(m+qn)}{\lambda(q)}, \end{equation*} and therefore $$0< h_d(m+qn)=\lambda(q)\ell-\lambda(q)\ell'=(x+\lambda(q)\ell)-(x+\lambda(q)\ell')<\lambda(q)L'/s \leq L/s.$$ Moreover, we know that $r_d+d(m+qn)=r_{qd}+qdn \in \P$, so $m+qn \in \Lambda_d$, and the result follows.
\end{proof}

\subsection{Proof of Theorem \ref{pint2}} Fixing $d\leq \max\{\log N, Q(\delta)\}$ and partitioning $[1,N]$ into arithmetic progressions of step size $\lambda(q_0)$ and length between $N/2\lambda(q_0)$ and $N/\lambda(q_0)$, the pigeonhole principle guarantees the existence of an arithmetic progression $P=\{x+\ell\lambda(q_0): 1\leq \ell \leq N_0\}$ such that $N/2\lambda(q_0) \leq N_0 \leq N/\lambda(q_0)$ and $|A\cap P|/N_0 \geq \delta$.

\noindent This allows us to define $A_0 \subseteq [1,N_0]$ by $A_0=\{\ell \in [1,N_0]: x+\ell\lambda(q_0) \in A \}$, which consequently satisfies \begin{equation*}|A_0|/N_0=\delta_0\geq \delta,  \quad N_0\geq N/Q(\delta)^{4kK}, \quad \text{and} \quad \R_{q_0d}(A_0)\leq \R_d(A), \end{equation*} where the last fact follows from Proposition \ref{scale}.

\noindent We then iteratively apply Lemma \ref{mainit}, which yields, for each $m$, a set $A_m \subseteq [1,N_m]$ with $|A_m|=\delta_mN_m$ and 
\begin{equation}\label{RAPh} \R_{d_m}(A_m) \leq \R_d(A) \end{equation}
satisfying
\begin{equation}\label{Ndeld} N_m \geq (c\delta)^{Cm}N_0, \quad \delta_m \geq \delta_{m-1}+c\delta_{m-1}^{\gamma}, \quad q_0 \mid d_m, \quad \text{and} \quad d_m/q_0 \leq (c\delta)^{-Cm} d
\end{equation} 
as long as 
\begin{equation} \label{Nmdel} N_m \geq \sqrt{N}, \quad d_m/q_0\leq \max\{\log N, Q(\delta)\}^2,
\end{equation}
and 
\begin{equation}\label{RAm} \R_{d_m}(A_m) \leq \delta_m^2N_m\Psi_{d_m}/8.
\end{equation}
By (\ref{Ndeld}), we see that the density $\delta_m$ would surpass $1$, and hence (\ref{Nmdel}) or (\ref{RAm}) must fail, with \begin{equation} \label{mbound} m=C\delta^{-(\gamma-1)}. \end{equation}

\noindent However,  if $C_0$ is sufficently large then (\ref{mbound}) implies $(c\delta)^{-Cm} \leq Q(\delta)$, hence $N_m\geq N/Q(\delta) \geq \sqrt{N}$ and $d_m/q_0 \leq Q(\delta)d \leq \max\{\log N, Q(\delta)\}^2$, so (\ref{Nmdel}) holds. Further, we see by (\ref{psiB2}) and (\ref{Ndeld}) that
\begin{equation*} \delta_m^2N_m\Psi_{d_m} \geq (c\delta)^{Cm}N_0 M_d / q_0 \geq \exp(-C\delta^{-(k+\epsilon-1)})NM_d,
\end{equation*}
so if $\R_d(A) \leq \exp(-C\delta^{-(k+\epsilon-1)})NM_d$ for a sufficiently large constant $C$, then by (\ref{RAPh}) we have that (\ref{RAm}) also holds. This yields a contradiction, and the theorem follows. 
\begin{equation*}
\end{equation*}

\section{Density Increment Strategy: Deducing Lemma \ref{mainit}}

\subsection{Fourier Analysis on $\Z$} We embed our finite sets in $\Z$, on which we utilize the discrete Fourier transform. Specifically, for a function $F: \Z \to \C$ with finite support, we define $\widehat{F}: \T \to \C$, where $\T$ denotes the circle  parameterized by the interval $[0,1]$ with $0$ and $1$ identified, by \begin{equation*} \widehat{F}(\alpha) = \sum_{x \in \Z} F(x)e^{-2 \pi ix\alpha}. \end{equation*}

\noindent Given $L\in \N$ and a set $B\subseteq [1,L]$ with $|B|=\sigma L$, we examine the Fourier analytic behavior of $B$ by considering the \textit{balance function}, $f_B$, defined by
\begin{equation*} f_B=1_B-\sigma 1_{[1,L]}.\end{equation*} 
\subsection{The Circle Method} We analyze the behavior of $\widehat{f_B}$ using the Hardy-Littlewood circle method, \quad \quad decomposing the frequency space into two pieces: the points on the circle which are close to rationals with small denominator, and those which are not.

\begin{definition}Given $L\in \N$ and $\eta>0$, we define, for each $q\in \N$ and $a\in[1,q]$,
\begin{equation*} \mathbf{M}_{a/q}=\mathbf{M}_{a/q}(L,\eta)=\left\{\alpha \in \T : |\alpha-\frac{a}{q} | < \frac{1}{\eta^{\gamma}L} \right\} \quad \text{and} \quad \mathbf{M}_q = \bigcup_{(a,q)=1} \mathbf{M}_{a/q}.\end{equation*} 
We then define $\mathfrak{M}$, the \textit{major arcs}, by 
\begin{equation*}\mathfrak{M}=\bigcup_{q=1}^{\eta^{-\gamma}} \mathbf{M}_q, \end{equation*} and $\mathfrak{m}$, the \textit{minor arcs}, by \begin{equation*} \mathfrak{m}= \T \setminus \mathfrak{M}. \end{equation*}
\end{definition}

\subsection{$L^2$ Concentration and Density Increment Lemmas}
As usual, the philosophy behind the argument is that a deficiency in the desired arithmetic structure from a set $B$ represents  nonrandom behavior, which should be detected in the Fourier analytic behavior of $B$. Specifically,  we follow the approach of Lyall and Magyar  \cite{LM} to locate one small denominator $q$ such that $\widehat{f_B}$ has $L^2$ concentration around rationals with denominator $q$, then use that information to find a long arithmetic progression on which $B$ has increased density. 

\begin{lemma}[$L^2$ Concentration] \label{L2} Suppose $B \subseteq [1,L]$, $|B|/L=\sigma \geq \delta$, and $L \geq \sqrt{N}$, and let $\eta=c_2\sigma$ for a sufficiently small constant $c_2>0$. Suppose further that 
\begin{equation*} q_0\mid d, \quad d/q_0 \leq \max\{\log N, Q(\delta)\}^2, \quad \text{and} \quad\R_d(B) \leq \sigma^2L\Psi_d/8. \end{equation*}
If  $|B \cap (L/9, 8L/9)| \geq 3\sigma L/4,$ then there exists $q \leq \eta^{-\gamma}$ such that 
\begin{equation*}  \int_{\mathbf{M}_q} |\widehat{f_B}(\alpha)|^2 \textnormal{d}\alpha \gg \sigma^{\gamma+1}L.\end{equation*}
\end{lemma}

\noindent We now invoke a variation of the usual $L^2$ density increment. Specifically, we quote a result which follows from Proposition 7.2  of \cite{Ruz}.

\begin{lemma}[Density Increment] \label{dinc} Suppose $B \subseteq [1,L]$ with $|B|=\sigma L$ and let $\eta=c_2\sigma$. If  \begin{equation*}  \int_{\mathbf{M}_q} |\widehat{f_B}(\alpha)|^2 \textnormal{d}\alpha \geq \omega \sigma^2L,\end{equation*} 
then there exists an arithmetic progression 
\begin{equation*} P=\{x+\ell \lambda(q) : 1\leq \ell \leq L'\} \end{equation*} with 
\begin{equation*} L/\lambda(q) \geq L' \gg  \min\{\eta^{\gamma},\omega\sigma\}L/\lambda(q) \quad \text{and} \quad |A \cap P|/L' \geq \sigma + \omega\sigma/4. \end{equation*}
\end{lemma}

\subsection{Proof of Lemma \ref{mainit}} Suppose $B \subset [1,L]$ meets all the hypotheses of the lemma. 

\noindent If $|B \cap (L/9, 8L/9)| < 3\sigma L/4$, then 
\begin{equation*} \max \left\{ |B\cap [1,L/9]|, |B\cap [8L/9]| \right\} \geq \sigma L/8. \end{equation*}
In other words, $B$ has density at least $9\sigma/8$ on one of these intervals.

\noindent Otherwise, Lemmas \ref{L2} and \ref{dinc} apply, so in either case there exists $q \leq \eta^{-\gamma}$ and an arithmetic progression  $P= \{x+\ell \lambda(q) : 1 \leq \ell \leq L'\} $ with 
\begin{equation*}L/\lambda(q) \geq L' \gg \sigma^{\gamma}L/\lambda(q))\gg \sigma^{\gamma(k+1)} L  \quad and \quad |B\cap P|/L' \geq  \sigma + c\sigma^{\gamma}. \end{equation*}
This allows us to define a new set $B' \subset [1,L']$ by
\begin{equation*}\label{A'} B' =\{\ell \in [1,L']: x+\ell\lambda(q)\in B\}, \end{equation*} which by Proposition \ref{scale} satisfies $\R_{qd}(B')\leq \R_{d}(B)$, as required.
\qed

\subsection{Proof of Lemma \ref{L2}}

Suppose $B \subseteq [1,L]$, $|B|/L=\sigma\geq\delta $, and $L \geq \sqrt{N}$. Let $\eta=c_2\sigma$, and suppose further that $q_0\mid d$ and $d/q_0 \leq \max\{\log N, Q(\delta)\}^2$. For the remainder of the proof, we keep this $d$ fixed and omit it from the notations $H_d,\ M_d, \ \nu_d, \ \R_d,$ and $\Psi_d$ defined in Sections 2 and 3. 

\noindent Since $h_d(H)\subseteq [1,L/9)$, we see that 
\begin{align*}
\sum_{\substack{x \in \Z \\ y\in H}} f_B(x)f_B(x+h_d(y))\nu(y)&=\sum_{\substack{x \in \Z \\ y\in H}} 1_B(x)1_B(x+h_d(y))\nu(y)-\sigma\sum_{\substack{x \in \Z \\ y\in H}} 1_B(x)1_{[1,L]}(x+h_d(y))\nu(y) \\\\ &-\sigma \sum_{\substack{x \in \Z \\ y\in H}} 1_{[1,L]}(x-h_d(y))1_B(x)\nu(y) +\sigma^2\sum_{\substack{x \in \Z \\ y\in H}} 1_{[1,L]}(x)1_{[1,L]}(x+h_d(y))\nu(y)  \\\\&\leq \R(B) + \Big(\sigma^2L-\sigma\Big(|B\cap [1,8L/9)|+|B\cap (L/9,L]|\Big)\Big) \sum_{y\in H} \nu(y).
\end{align*}
By (\ref{symdif}) we have that $$\sum_{y\in H} \nu(y)= \Psi + O(\log L),$$ so if $|B \cap (L/9, 8L/9)| \geq 3\sigma L/4$ and $\R(B) \leq \sigma^2L\Psi/8$, we have that 
\begin{equation}\label{neg} \sum_{\substack{x \in \Z \\ y\in H}} f_B(x)f_B(x+h_d(y))\nu(y) \leq -\sigma^2L\Psi/8. 
\end{equation} 
One can easily check using (\ref{symdif}) and orthogonality of characters that 
\begin{equation}\label{orth}
\sum_{\substack{x \in \Z \\ y\in H}} f_B(x)f_B(x+h_d(y))\nu(y)=\int_0^1 |\widehat{f_B}(\alpha)|^2S_M(\alpha)\textnormal{d}\alpha +O(L\log L),
\end{equation}
where 
\begin{equation*}S_X(\alpha)= \sum_{x=1}^X\nu(x)e^{2\pi i h_d(x)\alpha}.
\end{equation*}
Combining (\ref{neg}) and (\ref{orth}), we have 
\begin{equation} \label{mass}
\int_0^1 |\widehat{f_B}(\alpha)|^2|S_M(\alpha)|\textnormal{d}\alpha \geq \sigma^2L\Psi/16.
\end{equation}

\noindent It follows from  Lemma \ref{RS}, an observation of Lucier on auxiliary polynomials, and Theorem 4.1 of \cite{lipan} that
\begin{equation}\label{Smaj}|S_M(\alpha)| \ll q^{-1/\gamma}\Psi  \text{ for all } \alpha \in \mathbf{M}_{q}\subset \mathfrak{M},\end{equation}
and 
\begin{equation}\label{Smin} |S_M(\alpha)| \leq C\eta \Psi \leq \sigma \Psi/32\text{ for all } \alpha\in \mathfrak{m}, \end{equation} 
provided we choose $c_2<1/32C$. We discuss these estimates in more detail in Section \ref{M}.  
 
\noindent From (\ref{Smin}) and Plancherel's Identity, we have \begin{equation*}  \int_{\mathfrak{m}} |\widehat{f_B}(\alpha)|^2|S_M(\alpha)|\textnormal{d}\alpha \leq \sigma^2L\Psi/32, \end{equation*} which together with (\ref{mass}) yields \begin{equation}\label{majmass}  \int_{\mathfrak{M}}|\widehat{f_B}(\alpha)|^2|S_M(\alpha)|\textnormal{d}\alpha \geq \sigma^2L\Psi/32. \end{equation} 
By  (\ref{Smaj}) and (\ref{majmass}), we have \begin{equation*} \sigma^2L \ll  \Big(\sum_{q=1}^{\eta^{-\gamma}} q^{-1/\gamma} \Big) \max_{q\leq \eta^{-\gamma} }\int_{\mathbf{M}_q}|\widehat{f_B}(\alpha)|^2\textnormal{d}\alpha \ll \sigma^{-\gamma+1}  \max_{q\leq \eta^{-\gamma} }\int_{ \mathbf{M}_q} |\widehat{f_B}(\alpha)|^2\textnormal{d}\alpha,\end{equation*}  and the lemma follows.
\qed

\section{Major and Minor Arc Estimates: Proof of (\ref{Smaj}) and (\ref{Smin})} \label{M}
We remain in the setting of the proof of Lemma \ref{L2}, recalling all hypotheses and notation defined there. We first state some required estimates, which we use to deduce (\ref{Smaj}) and (\ref{Smin}), then we include the necessary proofs in Appendix \ref{A}.

\begin{lemma}\label{asym} If $Q(\delta)\geq \log N$ and $\alpha = a/q +\beta$ with $q\leq (q_0Q(\delta)^2)^K$, $(a,q)=1$, and $|\beta| < (q_0Q(\delta)^2)^K/L$, then
\begin{equation*} S_M(\alpha) = \frac{\phi(d)}{\phi(qd)} G(a,q) \int_1^M (1-\chi(r_d)(dx)^{\rho-1})e^{2\pi i h_d(x)\beta}\textnormal{d}x + O(Me^{-5K^2c_1\sqrt{\log N}}),
\end{equation*}
where 
\begin{equation*}G(a,q)= \sum_{\substack{\ell=0 \\ (r_d+d\ell,q)=1}}^{q-1} e^{2\pi i h_d(\ell)a/q}.
\end{equation*}
\end{lemma}

\begin{lemma}\label{asym2} If $Q(\delta)\leq \log N$ and $\alpha = a/q +\beta$ with $q\leq (q_0(\log N)^2)^{K}$, $(a,q)=1$, and $|\beta| < (q_0(\log N)^2)^{K}/L$, then
\begin{equation*} S_M(\alpha) = \frac{\phi(d)}{\phi(qd)} G(a,q) \int_1^M e^{2\pi i h_d(x)\beta}\textnormal{d}x + O(Me^{-c\sqrt{\log N}}).\end{equation*}
\end{lemma}

\noindent In Appendix \ref{A}, we exhibit how Lemma \ref{asym} follows from Lemma \ref{RS}, and Lemma \ref{asym2} follows from Lemma \ref{SW} in an analogous, more standard way.

\begin{lemma}\label{gs} Suppose $g(x)=a_0+a_1x+\cdots+a_kx^k \in \Z[x]$. If $W,b\in \Z$, $q\in \N$ and $(a,q)=1$, then
\begin{equation} \Big|\sum_{\substack{\ell=0 \\ (W\ell+b,q)=1}}^{q-1} e^{2\pi i g(\ell) a/q}\Big| \leq C^{\omega(q)}\Big(\gcd(\textnormal{cont}(g),q_1)\gcd(a_k,q_2)\Big)^{1/k} q^{1-1/k},
\end{equation} 
where $\omega(q)$ is the number of distinct prime factors of $q$, $q=q_1q_2$, $q_2$ is the maximal divisor of $q$ which is coprime to W,  and \begin{equation*} \textnormal{cont}(g) := \gcd(a_1, \dots, a_k). \end{equation*}
\end{lemma} 

\noindent \textit{Remark.} In the published version of this paper, the factor of $C^{\omega(q)}$ in Lemma \ref{gs} is incorrectly absent.
 
\noindent The statement of Lemma \ref{gs} indicates that we could lose control of the sum $G(a,q)$ if the coefficients of the auxiliary polynomials $h_d$ share larger and larger common factors. The following observation of Lucier ensures that this does not occur.
\begin{lemma}[Lemma 28 in \cite{Lucier}] \label{content} For every $d\in \N$, \begin{equation*} \textnormal{cont}(h_d) \leq |\Delta(h)|^{(k-1)/2}\textnormal{cont}(h), \end{equation*} where $\Delta(h)=a^{2k-2}\prod_{i\neq j} (\alpha_i-\alpha_j)^{e_ie_j}$ if $h$ factors over the complex numbers as $a(x-\alpha_1)^{e_1}\dots(x-\alpha_r)^{e_r}$ with all the $\alpha_i$'s distinct.
\end{lemma}  
\noindent While the statement of Lemma \ref{content} is pleasingly precise, we only use that cont$(h_d)$ is uniformly bounded in terms of the original polynomial $h$.
\begin{corollary}\label{Gaq} If $(a,q)=1$, then
\begin{equation*} |G(a,q)| \leq C^{\omega(q)}q^{1-1/k}, \end{equation*}
for some $C=C(h)$.
\end{corollary} 

\subsection{Proof of (\ref{Smaj})} We treat the case of $Q(\delta)\geq \log N$ using Lemma \ref{asym}, and the other case follows in a similar, slightly simpler fashion from Lemma \ref{asym2}. Since $\eta^{-\gamma}<Q(\delta)$, the hypotheses of Lemma \ref{asym} are certainly satisfied whenever $\alpha \in \mathbf{M}_q$ with $q\leq \eta^{-\gamma}$. Therefore, Lemma \ref{asym} and Corollary \ref{Gaq}, combined with the bound 
\begin{equation}\Big|\int_1^M(1-\chi(r_d)(dx)^{\rho-1})e^{2\pi i h_d(x)\beta}\textnormal{d}x\Big| \leq M-\chi(r_d)(dM)^{\rho}/d\rho  \ll \Psi
\end{equation}
from Lemma \ref{RS} and the well known facts $\phi(qd)\geq \phi(q)\phi(d)$ and $$\phi(q) \geq  c_{\mu}q^{1-\mu}, \quad C^{\omega(q)}\leq C'_{\mu}q^{\mu} \quad \text{for any} \quad \mu>0,$$ yield  
\begin{equation*}|S_M(\alpha)| \ll q^{-1/\gamma}\Psi + O(Me^{-5K^2c_1\sqrt{\log N}}).
\end{equation*}
Finally, the lower bound 
\begin{equation}\Psi \gg M/q_0 \geq Me^{-3Kc_1\sqrt{\log N}}
\end{equation}
given by (\ref{psiB2}) and (\ref{assmp}) ensures that the error term is negligible, and the estimate follows.
\qed

\noindent For our minor arc estimate we need the following analog of Weyl's Inequality, due to Li and Pan, which generalizes work of Vinogradov.
\begin{lemma}\label{min} Suppose $g(x)=a_0+a_1x+\cdots+a_kx^k \in \Z[x]$ with $a_k>0$, $D, W\in \N$, and $b\in \Z$. If $U\geq \log D$,
$a_k \gg |a_{k-1}| + \cdots +|a_0|,$ and $W,|b|,a_k \leq U^k$, then 
\begin{equation*}\sum_{\substack{x=1 \\ Wx+b \in \P}}^D \log(Wx+b)e^{2\pi i g(x) \alpha} \ll \frac{D}{U}+U^CD^{1-c}
\end{equation*} 
for some constants $C=C(k)$ and $c=c(k)>0$, provided \begin{equation*} |\alpha -a/q| < q^{-2} \quad \text{for some} \quad U^{K} \leq q \leq g(D)/U^{K} \quad \text{and} \quad (a,q)=1.
\end{equation*}
\end{lemma}
\noindent Lemma \ref{min} is a rougher, only nominally generalized version of Theorem 4.1 of \cite{lipan}. That result restricts to the case where $U$ is a power of $\log D$, and provides a more precise bound in place of $K$, but the main achievement of the theorem is that one can take $U$ to be that small. Larger values of $U$, and hence stricter conditions on $q$, actually make the proof, which can be found in the appendix of that paper, slightly easier. Specifically, one can observe that the precise condition on $q$ is not utilized until Lemmas 4.11 and 4.12, and adaptations of those lemmas are sufficient to adapt the proof of Theorem 4.1.

\subsection{Proof of (\ref{Smin})} Again, we only treat the case $Q(\delta)\geq \log N$. For a fixed $\alpha \in \mathfrak{m}$, we have by the pigeonhole principle that there exist 
\begin{equation*}1\leq q \leq L/(q_0Q(\delta)^2)^K\end{equation*}
 and $(a,q)=1$ with \begin{equation*} |\alpha-a/q| < (q_0 Q(\delta)^2)^K/(qL). \end{equation*}
If $\eta^{-\gamma} \leq q \leq (q_0Q(\delta)^2)^K$, then $\alpha$ meets the hypotheses of Lemma \ref{asym}, and by reasoning identical to the proof of (\ref{Smaj}) we have 
 \begin{equation*} |S_M(\alpha)| \ll q^{-1/\gamma}\Psi \ll \eta \Psi.
 \end{equation*}
If $(q_0 Q(\delta)^2)^K \leq q \leq L/(q_0 Q(\delta)^2)^K$, then we can apply Lemma \ref{min} with $U=q_0 Q(\delta)^2$, along with (\ref{psiB2}) and the fact that $\eta > Q(\delta)^{-1}$ to conclude 
\begin{equation*} |S_M(\alpha)| \ll \frac{M}{q_0Q(\delta)^2}\ll \frac{\Psi}{Q(\delta)^2}<\eta \Psi, 
\end{equation*} as required. 
  
\noindent If $1\leq q \leq \eta^{-\gamma}$, then, letting $\beta=\alpha-a/q$, it must be the case that \begin{equation} \label{beta} |\beta|>\eta^{-\gamma}/L,\end{equation} as otherwise we would have $\alpha \in \mathfrak{M}$. By Lemma \ref{asym} it suffices to show
\begin{equation*}\Big|\int_1^M (1-\chi(r_d)(dx)^{\rho-1})e^{2\pi i h_d(x)\beta}\textnormal{d}x\Big| \ll \eta\Psi.
\end{equation*}
From (\ref{beta}) and Lemma 2.8 of \cite{vaughan}, we have for any $x>1$ that
\begin{equation*}\Big|\int_1^x  e^{2\pi i h_d(y)\beta}\textnormal{d}y \Big| \ll (b_d|\beta|)^{-1/k} \ll \eta M,
\end{equation*} hence by integration by parts, Lemma \ref{RS}, and (\ref{psiB2}) we see 
\begin{align*}\Big|\int_1^M (1-\chi(r_d)(dx)^{\rho-1})e^{2\pi i h_d(x)\beta}\textnormal{d}x\Big| &\ll \eta (M-\chi(r_d)(dM)^{\rho}/d) \\& \leq \eta\Big( M - \frac{\chi(r_d)(dM)^{\rho}}{d\rho} +2(1-\rho)M \Big) \ll \eta \Psi,
\end{align*} 
and the estimate is complete. 

\appendix 
\section{Exponential Sum Estimates: Proof of Lemma \ref{asym}, Lemma \ref{gs}, and Corollary \ref{Gaq}} \label{A}
\subsection{Proof of Lemma \ref{asym}} Fixing $q\leq (q_0Q(\delta)^2)^K$ and $(a,q)=1$, we first investigate the values of $S_X(a/q)$ for $X \geq N^{1/10k}$. We see that 
\begin{equation}\label{Sx} S_X(a/q)=\sum_{x=1}^X\nu(x)e^{2\pi i h_d(x)a/q} =\sum_{\ell=o}^{q-1} e^{2\pi i h_d(\ell)a/q} \sum_{\substack{x=1 \\ x \equiv \ell \text{ mod } q}}^X \nu(x),
\end{equation} 
and we note that
\begin{equation}\label{psi1} \sum_{\substack{x=1 \\ x \equiv \ell \text{ mod } q}}^X \nu(x) = \phi(d)\psi(dX+r_d, r_d+d\ell, qd)/d.
\end{equation}
Since $(r_d,d)=1$, we have that $(r_d+d\ell,qd)=1$ if and only if $(r_d+d\ell,q)=1$. Therefore, if $(r_d+d\ell,q)>1$, we have $\psi(dX+r_d, r_d+d\ell, qd) \leq \log(dX+r_d) \ll \log X$, whereas if $(r_d+d\ell, q)=1$, we have by (\ref{psi1}) and Lemma \ref{RS} that 
\begin{equation}\label{psi} \sum_{\substack{x=1 \\ x \equiv \ell \text{ mod } q}}^X \nu(x) = \frac{\phi(d)}{\phi(qd)}\Big(X-\chi(r_d)(dX)^{\rho}/d\rho\Big) + O(Xe^{-30kK^2c_1\sqrt{\log X}}).
\end{equation}
Combining (\ref{Sx}) and (\ref{psi}), we have 
\begin{equation*} S_X(a/q)=\frac{\phi(d)}{\phi(qd)}G(a,q)\Big(X-\chi(r_d)(dX)^{\rho}/d\rho\Big) + O(qXe^{-30kK^2c_1\sqrt{\log X}})
\end{equation*} for all $X \geq N^{1/10k}$. In particular, since $q\leq e^{5K^2c_1\sqrt{\log N}}$ and $M\gg N^{1/4k}$, we can apply trivial bounds for small values of $X$ and conclude 
\begin{equation}\label{SX} S_X(a/q)=\frac{\phi(d)}{\phi(qd)}G(a,q)\Big(X-\chi(r_d)(dX)^{\rho}/d\rho\Big)+ O(Me^{-10K^2c_1\sqrt{\log N}})
\end{equation}
for all $X \leq M$. 

\noindent Now suppose $\alpha=a/q+\beta$ with $|\beta| < (q_0Q(\delta)^2)^K/L$. By (\ref{SX}) and two applications of integration by parts, we see  
\begin{align*}S_M(\alpha) &= \sum_{x=1}^M \nu(x)e^{2\pi i h_d(x)a/q}e^{2\pi i h_d(x)\beta} \\\\ &=S_M(a/q)e^{2\pi i h_d(M)\beta} - \int_{1}^{M} S_x(a/q)2\pi i\beta h_d'(x)e^{2\pi i h_d(M)\beta}\text{d}x \\\\&= S_M(a/q)e^{2\pi i h_d(M)\beta}-\frac{\phi(d)}{\phi(qd)}G(a,q) \int_{1}^{M}\Big(x-\chi(r_d)(dx)^{\rho}/d\rho\Big) 2\pi i\beta h_d'(x)e^{2\pi i h_d(x)\beta}\text{d}x\\\\ &\quad +O((1+\beta L)Me^{-10K^2c_1\sqrt{\log N}}) \\\\&=\frac{\phi(d)}{\phi(qd)}G(a,q)\int_{1}^M\Big(1-\chi(r_d)(dx)^{\rho-1}\Big)e^{2\pi i h_d(x)\beta}\text{d}x +O(Me^{-5K^2c_1\sqrt{\log N}}),\\
\end{align*}
and the asymptotic is established.\\
\qed

\subsection{Proof of Lemma \ref{gs}} Fix $g, W, b,a,q$ as in Lemma \ref{gs}. We primarily make use of the well-known complete Gauss sum estimate
\begin{equation}\label{fgs} \Big|\sum_{\ell=0}^{q-1} e^{2\pi i g(\ell) a/q}\Big| \ll \gcd(\textnormal{cont}(g),q)^{1/k}q^{1-1/k},
\end{equation}
which can be found for example in Lemma 6 of \cite{Lucier}. As is often the case with this type of sum, we can simplify our argument by taking advantage of multiplicativity. Specifically, it is not difficult to show that if $q=q_1q_2$ with $(q_1,q_2)=1$, then 
\begin{equation*}\sum_{\substack{\ell=0 \\ (W\ell+b,q)=1}}^{q-1} e^{2\pi i g(\ell) a/q}= \Big(\sum_{\substack{\ell_1=0 \\ (W\ell_1+b,q_1)=1}}^{q_1-1} e^{2\pi i g(\ell_1) a_1/q_1}\Big)\Big( \sum_{\substack{\ell_2=0 \\ (W\ell_2+b,q_2)=1}}^{q_2-1} e^{2\pi i g(\ell_2) a_2/q_2}\Big),
\end{equation*} where $a/q=a_1/q_1+a_2/q_2$, so we can assume without loss of generality that $q=p^{j}$ for some $p\in \P$, $j \in \N$, and separately consider the cases of $p \mid W$ and $p \nmid W$.

\noindent If $p \mid W$ and $p \mid b$, then $W\ell+b$ is never coprime to $p^{j}$, so the sum is clearly zero. If $p \mid W$ and $p \nmid b$, then $W\ell+b$ is always coprime to $p^{j}$, so the sum is complete and the result follows from (\ref{fgs}). 

\noindent If $p \nmid W$, then $p \mid W\ell+b$ if and only if $\ell\equiv -bW^{-1}$ mod $p$. Therefore, 
\begin{equation} \sum_{\substack{\ell=0 \\ p\nmid W\ell+b}}^{p^{j}-1} e^{2\pi i g(\ell) a/p^{j}}= \sum_{\ell=0}^{p^{j}-1} e^{2\pi i g(\ell) a/p^{j}}- \sum_{r=0}^{p^{j-1}-1} e^{2\pi i g(pr+m) a / p^{j}},
\end{equation} where $m \equiv -bW^{-1}$ mod $p$, and by (\ref{fgs}) we need only obtain the estimate for the second sum.

\noindent Setting 
\begin{equation*} \tilde{g}(r)=\frac{g(pr+m)-g(m)}{p},
\end{equation*}
we see that $\tilde{g}$ is a polynomial with integer coefficients and leading coefficient $a_kp^{k-1}$. In particular, 
\begin{equation*}\gcd(\text{cont}(\tilde{g}),p^{j-1}) \leq p^{k-1}\gcd(a_k,p^{j-1}).\end{equation*} Therefore, by (\ref{fgs}) we have
\begin{align*}\Big|\sum_{r=0}^{p^{j-1}-1} e^{2\pi i g(pr+m) a / p^{j}}\Big|&= \Big|\sum_{r=0}^{p^{j-1}-1}e^{2\pi i (g(pr+m)-g(m)) a / p^{j}}\Big| \\\\ &= \Big|\sum_{r=0}^{p^{j-1}-1}e^{2\pi i\tilde{g}(r)a/p^{j-1}} \Big| \\\\& \ll \Big(p^{k-1}\gcd(a_k,p^{j-1})\Big)^{1/k}p^{(j-1)(1-1/k)} \\\\&\leq \gcd(a_k,p^{j})^{1/k}p^{j(1-1/k)},
\end{align*}
as required.\\
\qed

\subsection{Proof of Corollary \ref{Gaq}} From its definition, we see that the leading coefficient of $h_d$ is $d^kb/\lambda(d)$, where $b$ is the leading coefficient of $h$. Given $q\in \N$ and $(a,q)=1$, we write $q=q_1q_2$, where $q_2$ is the maximal divisor of $q$ which is coprime to $d$. In particular, 
\begin{equation} \label{bk} \gcd(d^kb/\lambda(d),q_2) \leq b.
\end{equation}
Therefore, by Lemma \ref{gs}, Lemma \ref{content}, and (\ref{bk}) we have 
\begin{align*}|G(a,q)| = \Big|  \sum_{\substack{\ell=0 \\ (r_d+d\ell,q)=1}}^{q-1} e^{2\pi i h_d(\ell)a/q}\Big| &\ll \Big( \gcd(\text{cont}(h_d),q_1)b \Big)^{1/k}q^{1-1/k} \ll q^{1-1/k},
\end{align*}
and all the required estimates are established.\\
\qed

\section{Theorems \ref{pintQ} and \ref{rel}: An Informal Discussion} \label{C}
Using the result of Theorem \ref{pint2}, one can almost immediately conclude Theorem \ref{rel} by replicating the transference principle argument used in \cite{Le} to obtain Theorem F from a uniform version of Theorem C. Similarly, using weighted analogs of the major and minor arc estimates from this paper, one can almost immediately conclude Theorem \ref{pintQ} by reproducing the method of \cite{HLR} used to prove Theorem D. In each instance, there are a few issues that arise and we address in this section, which is best read in conjunction with those two papers.  First, we recall that for the arguments in \cite{HLR} and \cite{Le}, it is convenient, if not necessary, to do analysis with a discrete frequency domain, that is to embed subsets of $[1,N]$ into the finite group $\Z/N\Z$ as opposed to the integers. 

\subsection{Higher Moments of Weyl Sums.} \label{HM} To adapt the methods of \cite{HLR} and \cite{Le}, we need analogous estimates on higher moments of weighted and unweighted exponential sums over polynomials in primes. Specifically, if we borrow some notation from Section \ref{prelim} and define $$T(\alpha)=\frac{1}{\Psi_d}\sum_{x\in H_d}\nu_d(x)e^{2\pi i h_d(x)\alpha}\quad \text{and} \quad W(\alpha)= \frac{M_d}{\Psi_dN} \sum_{x\in H_d}\nu_d(x)h_d'(x)e^{2\pi i h_d(x)\alpha},$$ then it is straightforward to apply the major and minor arc estimates from this paper, weighted analogs thereof, and higher moment estimates on standard Weyl sums (see \cite{Wooley} for example) to conclude under appropriate conditions that
$$ \sum_{t\in \Z/N\Z} |T(t/N)|^s = N \int_0^1 |T(\alpha)|^s \textnormal{d}\alpha \ll 1 \quad \text{and} \quad  \sum_{t\in \Z/N\Z} |W(t/N)|^s = N \int_0^1 |W(\alpha)|^s \textnormal{d}\alpha \ll 1.$$ It is with these estimates in mind that we chose $s=2^k+6$, although something much smaller would suffice, and the above equalities follow from the dependence on $s$ in the definition of $H_d$, as the relevant mod $N$ congruences imply equality.

\subsection{Applying Lemma \ref{RS} to Theorem \ref{pintQ}.} Because the method of \cite{HLR} does not involve estimating the number of solutions to the desired equation, it suffices for the proof of Theorem \ref{pintQ} to use a simplified form of Lemma \ref{RS} in which $Q(\delta)$ is replaced with $e^{c_1\sqrt{\log N}}$ throughout. In order to obtain a usable analog to Lemma 1 of \cite{HLR}, we need to initially pass to a subprogression of step size $\lambda(q_0)$ and replace the condition $d\leq N^{.01}$ with $d \leq e^{c\sqrt{\log N}}$ for a sufficiently small constant $c$. This requires us to replace the $L^2$ concentration upper bound $\sigma^2(\log N)^{-1+\epsilon}$ with $\sigma^2(\log N)^{-\frac{1}{2}+\epsilon}$, which is the  reason for the factor of $2$ discrepancy between Theorem D and Theorem \ref{pintQ}.

\subsection{``Square Root Cancellation" in Theorem \ref{pintQ}} The proof of Theorem D intimately uses the fact that for a quadratic polynomial, the normalized, weighted Weyl sum has ``square root cancellation" on the major arcs. In our setting, we can apply weighted analogs of Lemmas \ref{asym} and \ref{asym2} to conclude under appropriate conditions that if $t/N$ is close to a rational $a/q$ with $(a,q)=1$, then $$W(t/N) \ll \frac{q^{1/2}}{\phi(q)}\min\{1,(N|t/N-a/q|)^{-1} \}  \ll q^{-1/2}\log\log q\min\{1,(N|t/N-a/q|)^{-1} \},$$ where $W$ is as in Section \ref{HM} and the last inequality is a standard estimate on $\phi$. While this is not quite as good as the estimate used in the proof of Lemma 2 of \cite{HLR}, the error of $\log\log q$ can easily be absorbed with other negligible terms as in the end of that proof  (in fact $\log q$ would be fine as well). For a more detailed proof of Theorem \ref{pintQ}, see \cite{thesis}.

\subsection{Rephrasing Theorem \ref{pint2} to Deduce Theorem \ref{rel}.} Theorem \ref{pint2} implies the following, less precise statement, which uses notation defined in Section \ref{prelim} and  is ready-made for applying a transference principle. 
\begin{theorem} \label{transready} If $h\in\Z[x]$ is a $\P$-intersective polynomial and $F:\Z/N\Z \to [0,1]$ with  
$$\frac{1}{N}\sum_{x\in \Z/N\Z} F(x) \geq \delta >0,$$ then there exist  constants $c(h,\delta)>0$ and $N_0(h,\delta)$ such that 
$$\frac{1}{NM_d}\sum_{\substack{x\in \Z/N\Z \\ y\in H_d}}F(x)F(x+h_d(y))\nu_d(y) \geq c(h,\delta)$$ provided $d\leq \log N$ and $N\geq N_0(h,\delta)$.
\end{theorem}

\noindent Once armed with Theorem \ref{transready} and the unweighted higher moment estimate from Section \ref{HM}, Theorem \ref{rel} follows in the identical fashion that Theorem F follows from a uniform version of Theorem C, as in \cite{Le}.

\end{document}